\documentclass[12pt]{article}
\usepackage{graphicx, amsmath, amssymb, amsthm,anysize,float,framed,paralist,booktabs}
\usepackage{algorithm,algorithmic,multicol}
\usepackage{hyperref}
\usepackage{booktabs,todonotes}
\hypersetup{citecolor=red, linkcolor=blue, colorlinks=true}

\marginsize{2cm}{2cm}{1cm}{3cm}

\newcommand{\gauss}[2]{\genfrac{[}{]}{0pt}{}{#1}{#2}}
\newcommand{\gaussm}[1]{[#1]}

\newcommand{\scrC}{\mathcal{C}}

\newcommand{\scrH}{\mathcal{H}}

\newcommand{\bbP}{\mathbb{P}}

\newcommand{\bbF}{\mathbb{F}}

\newcommand{\Sp}{Sp}
\newcommand{\id}{\text{id}}

 \numberwithin{equation}{section}
 \newtheorem{theorem}[equation]{Theorem}
 \newtheorem{cor}[equation]{Corollary}
 \newtheorem{construction}[equation]{Construction}
 \newtheorem{lemma}[equation]{Lemma}
\newtheorem{question}{Question}

 \theoremstyle{definition}
 
 \theoremstyle{rem}
 \newtheorem{rem}[equation]{Remark}

\title
 {A Switching for all Strongly Regular Collinearity Graphs From Polar Spaces}
\author{Ferdinand Ihringer}

\begin{document}

\maketitle

\begin{abstract}
  We describe a general construction of strongly regular graphs from the collinearity graph of a finite classical polar spaces of rank at least $3$
  over a finite field of order $q$. 
  We show that these graphs are non-isomorphic to the collinearity graphs and have the same parameters.
  To our knowledge for most of these parameters these graphs are new as the collinearity graphs were the only known examples.
\end{abstract}



\section{Introduction}

Recently, many researchers constructed strongly regular graphs with the same parameters 
as the collinearity graphs of finite classical polar spaces. This was triggered by a result
of Abiad and Haemers, who used Godsil-McKay switching \cite{Godsil1982} to obtain strongly regular graphs with the same
parameters as $\Sp(2d, 2)$ for all $d \geq 3$ \cite{Abiad2016}. More new graphs with the same parameters
as $\Sp(2d, 2)$ or related graphs were found by various other researchers \cite{Hui2016,Kubota2016}. Barwick, Jackson and Penttila
generalized some of these constructions to the graphs $O^+(2d, 2)$, $O(2d+1, 2)$ and $O^-(2d+2, 2)$ \cite{Barwick2016}.
There is also a construction by Kantor for collinearity graphs from finite classical polar spaces where, 
when the construction is applicable, the non-isomorphism to existing graphs does not seem to be known in general \cite{Kantor1982a}.
We construct new strongly regular graphs with the same parameters as all collinearity graphs that come from finite classical polar
spaces of rank at least $3$; that is $\Sp(2d, q)$, $O^+(2d, q)$, $O(2d+1, q)$, $O^-(2d+2, q)$, $U(2d, q)$ and $U(2d+1, q)$ (Construction \ref{constr:main}).
Unlike all the previous constructions, which use a variant of switching, $q$ can be any prime power.
Also we do not rely on the existence of spreads, which often do not exist \cite[Table 3]{DeBeule2012}, but which are key for Kantor's construction.
This implies that most of the constructed graphs are new. Our construction makes it apparent that there 
are most likely many different strongly regular graphs with these parameters.

For our result we modify Godsil-McKay switching for these particular graphs. The modification is heavily inspired by a work by Jungnickel and Tonchev
who did something similar for designs coming from projective spaces \cite{Jungnickel2010}.

This work is structured as follows. In Section \ref{sec:prelim} we introduce our notation and basic properties of
finite classical polar spaces. In Section \ref{sec:switching} we provide our construction for new strongly regular graphs.
In Section \ref{sec:noniso}, we show that our construction yields non-isomorphic graphs.
We conclude with some questions for further work in Section \ref{sec:openproblems}.

\section{Preliminaries}\label{sec:prelim}

We repeat several important facts about projective, affine and polar spaces and fix the used notation.
We refer to \cite{Hirschfeld1979,Hirschfeld1991} as a standard reference on the topic.

\subsection{Projective and Affine Spaces}

We only use classical projective spaces over finite fields. Let $\bbF_q$ denote
the finite field of order $q$. We identify the projective space of (projective) dimension $d-1$
over $\bbF_q$ with $\bbF_q^d$:
the $1$-dimensional subspaces of $\bbF_q^d$ are the \textit{points} and the $2$-dimensional
subspaces of $\bbF_q^d$ are the \textit{lines}. We denote the number of $t$-dimensional
subspaces of $\bbF_q^d$ by $\gauss{d}{t}_q$. When $q$ is clear from the context, then we 
write $\gauss{d}{t}$ instead of $\gauss{d}{t}_q$. A standard argument \cite[Section 3.1]{Hirschfeld1979} shows
\begin{align*}
  \gauss{d}{t} = \prod_{i=1}^t \frac{q^{d-i+1}-1}{q^i-1}.
\end{align*}
Notice that
\begin{align*}
  \gauss{d}{t} = \gauss{d}{d-t}.
\end{align*}
We denote $\gauss{d}{1}$, the number of points and the number of hyperplanes, by $\gaussm{d}$.
Also recall that $\gaussm{d-s}$ is the number of hyperplanes through an $s$-dimensional subspace of $\bbF_q^d$.

One can obtain an affine space $\tilde{P}$ of dimension $d-1$ from a projective space $P$ of dimension $d-1$
by removing one hyperplane $L$ of $\bbF_q^d$ (\textit{hyperplane at infinity}) \cite[Section 2.2]{Hirschfeld1979}. 
Then the $q^{d-1}$ points ($1$-dimensional subspaces)
of $P \setminus L$ correspond to the points of the affine space $\tilde{P}$.
Let $H$ and $H'$ be two different hyperplanes of $P$. Let $L'$ denote the intersection of $H$ and $H'$.
Then $L'$ is a $(d-2)$-dimensional subspace of $P$. If $L' \subseteq L$, then $H$ and $H'$ are disjoint
affine hyperplanes of $\tilde{P}$ and they lie in the same \textit{parallel class (of hyperplanes)}.
Otherwise they intersect in an affine subspace of dimension $d-2$ with $q^{d-2}$ points.
Notice that we have as many parallel classes of hyperplanes as we have $(d-2)$-dimensional
subspaces of $L$. That are $\gaussm{d-1}$. Each parallel class contains $q$ hyperplanes
and these partition the points $\tilde{P}$.

\subsection{Polar Spaces}

We start with the definition of the finite classical polar spaces. We refer to \cite{Hirschfeld1991}, \cite[Chapter 8]{Tits1974} and \cite[Section 9.5]{Brouwer1989} for
details. Notice that there are six types of finite classical polar spaces and we define them as the incidence structure of 
subspaces of a finite vector space $V$ which vanish on a given sesquilinear or quadratic form. We call the vanishing subspaces \textit{isotropic}.
For a given polar space $\bbP$, we sometimes say subspace of $\bbP$ instead of isotropic subspace.

In the following we give standard forms for the finite classical polar spaces and specify $V$ for each case.
We associate a parameter $e \in \{ 0, 1/2, 1, 3/2, 2\}$ with each polar space. The parameter $d$ denotes the 
maximal dimension of an isotropic subspace of the polar space. Note that all maximal isotropic subspaces
have dimension $d$. We call subspaces of dimension $d$ \textit{generators}.

\begin{enumerate}[i.]
 \item $O^+(2d, q)$: the vector space is $\bbF_q^{2d}$. The quadratic form is $x_1^2 + x_2x_3 + \ldots + x_{2d-1} x_{2d}$. Here $e=0$.
 \item $O(2d+1, q)$: the vector space is $\bbF_q^{2d+1}$. The quadratic form is $x_1x_2 + x_3x_4 + \ldots + x_{2d} x_{2d+1}$. Here $e=1$.
 \item $O^-(2d+2, q)$: the vector space is $\bbF_q^{2d+2}$. The quadratic form is $f(x_1, x_2) + x_3x_4 + \ldots + x_{2d+1} x_{2d+2}$. Here $e=2$ and 
      $f$ is a homogenous quadratic polynomial that is irreducible over $\bbF_q$.
 \item $\Sp(2d, q)$: the vector space is $\bbF_q^{2d}$. The symplectic form is $x_1y_2 - x_2y_1 + \ldots + x_{2d-1}y_{2d} - x_{2d}y_{2d-1}$. Here $e=1$.
 \item $U(2d, q)$: here $q$ is a square and the vector space is $\bbF_q^{2d}$. The associated Hermitian form is $x_1y_1^{\sqrt{q}}+\ldots+x_{2d}y_{2d}^{\sqrt{q}}$.
    Here $e=1/2$.
 \item $U(2d+1, q)$: here $q$ is a square and the vector space is $\bbF_q^{2d+1}$. The associated Hermitian form is $x_1y_1^{\sqrt{q}}+\ldots+x_{2d+1}y_{2d+1}^{\sqrt{q}}$.
    Here $e=3/2$.
\end{enumerate}

For a subspace $S$ of a polar space $\bbP$, $S^\perp$ denotes the incidence structure of all the subspaces $T$ of $\bbP$ such that
$T$ is perpendicular to $S$, i.e. $\langle S, T \rangle$ is a subspace of $\bbP$. 
The \textit{collinearity graph} $\Gamma_0$ of a polar space has the $1$-dimensional subspaces of $\bbP$, \textit{points}, as vertices
and two different points $x$ and $y$ are adjacent if $x \in y^\perp$. For a given polar spaces, it is well-known that $\Gamma_0$ is strongly regular with parameters\footnote{See Lemma 2.3.1 in the Master's thesis of the author: \url{http://math.ihringer.org/mscthesis/thesis.pdf}}
 $(v_0, k_0, \lambda_0, \mu_0)$ where
\begin{align*}
  &v_0 = (q^{d-1+e}+1) \gaussm{d},\\
  &k_0 = q (q^{d-2+e}+1) \gaussm{d-1},\\
  &\lambda_0 = q-1 + q^2 (q^{d-3+e}+1) \gaussm{d-2},\\
  &\mu_0 = (q^{d-2+e}+1) \gaussm{d-1}.
\end{align*}

In the following we always fix one of the six given types of finite classical polar spaces.
Furthermore, we fix $q$ and $e$. Hence, we write $\bbP_d$
for a polar space of rank $d$ without specifying $q$, $e$ or the type of $\bbP_d$.

As for projective and affine spaces, we repeat some basic facts about finite classical polar spaces.
Notice that the counting arguments for quadrics in \cite[Chapter 22]{Hirschfeld1991}, which we use here,
work the same for the other finite classical polar spaces.

\begin{lemma}[{\cite[Section 22.4, Section 23.3]{Hirschfeld1991}, \cite[Chapter 8]{Tits1974}}]\label{lem:gen_max_dimension}
  An isotropic subspace of a polar space of rank $d$ has at most dimension $d$.
\end{lemma}

\begin{lemma}[{\cite[Section 8.5.2]{Tits1974}}]\label{lem:basic_quotient_geom}
  For an $s$-dimensional subspace $S$ of $\bbP_d$, the quotient geometry $S^\perp/S$ is isomorphic to $\bbP_{d-s}$.
\end{lemma}

\begin{lemma}\label{lem:gens_on_comax}
  A $(d-1)$-dimensional subspace of $\bbP_d$ lies in $q^e+1$ generators.
\end{lemma}
\begin{proof}
  By Lemma \ref{lem:basic_quotient_geom}, the number of generators through a $(d-1)$-dimensional subspace
  corresponds to the number of points in $\bbP_1$. This is $q^e+1$.
\end{proof}

\begin{lemma}[{\cite[Section 26.1]{Hirschfeld1991}}]\label{lem:basic_perp_property}
  Let $p$ be a point of $\bbP_d$ and $S$ an $s$-dimensional subspace of $\bbP_d$. Then
  $S \subseteq p^\perp$ or $S \cap p^\perp$ is a $(s-1)$-dimensional subspace of $S$.
\end{lemma}

\begin{lemma}\label{lem:Hxyz_pw_dis}
  Let $P$ be a generator of $\bbP_d$ and let $\ell$ be a line of $\bbP_d$ that meets $P$ trivially.
  Then $\ell^\perp \cap P$ is a $(d-2)$-dimensional subspace of $P$ and $\ell^\perp \cap P \subseteq x^\perp \cap P \nsubseteq y^\perp \cap P$
  for all points $x, y \in \ell$, $x \neq y$.
\end{lemma}
\begin{proof}
  By Lemma \ref{lem:basic_perp_property}, $\ell^\perp \cap P$ is a subspace of at least dimension $d-2$.
  The isotropic subspace $\langle \ell, \ell^\perp \cap P \rangle$ satisfies
  \begin{align*}
    d &\geq \dim(\langle \ell, \ell^\perp \cap P \rangle) \\
    &= \dim(\ell) + \dim(\ell^\perp \cap P) - \dim(\ell \cap \ell^\perp \cap P)\\
    &= \dim(\ell) + \dim(\ell^\perp \cap P) = 2 + \dim(\ell^\perp \cap P).
  \end{align*}
  Hence, $\ell^\perp \cap P$ has dimension $d-2$.
  Clearly, for all points $x \in \ell$ we have $\ell^\perp \subseteq x^\perp$.
  This implies $\ell^\perp \cap P \subseteq x^\perp \cap P$.
  As $\ell^\perp = \langle x, y \rangle^\perp$, but $x^\perp \cap P$ and $y^\perp \cap P$ are hyperplanes of $P$ by Lemma \ref{lem:basic_perp_property},
  $x^\perp \cap P \neq y^\perp \cap P$.
\end{proof}

\section{Subgeometry Switching}\label{sec:switching}

In this section we describe a switching operation for the collinearity graphs
of polar graphs that creates new strongly regular graphs with the same parameters.
Godsil-McKay switching needs a subset $Y$ of vertices such that 
$Y$ is regular and all vertices not in $Y$ are adjacent to $0$, $|Y|/2$ or $|Y|$ vertices of $Y$.
In the following we describe a switching operation (limited to the considered graphs) for a specific subset $Y$ where all vertices
not in $Y$ are adjacent to $0$, $|Y|/m$ or $|Y|$ vertices of $Y$ for some integer $m \geq 2$.
For $m=2$ one case of our switching corresponds to a Godsil-McKay switching of $\Gamma_0$ (see Remark \ref{rem:switching}).

\begin{construction}\label{constr:main}
Let $X$ be the point set of a polar space $\bbP_d$ (of rank $d$).
Let $L$ be a subspace of $\bbP_d$ of dimension $d-1$.
Then $L$ lies in $q^e+1$ subspaces $P_1, P_2, \ldots, P_{q^e+1}$ of $\bbP_d$ of dimension $d$  by Lemma \ref{lem:gens_on_comax}.
In the following we slightly abuse notation and write $P_i$ without specifying that $i \in \{ 1,\ldots, q^e+1 \}$.
Each $P_i \setminus L$ can be identified with the affine geometry of dimension $d-1$ over $\bbF_q$.
We write $\tilde{P}_i$ for the affine geometry $P_i \setminus L$.
We partition $X$ as follows:
\begin{enumerate}[i.]
 \item the points of $L$,
 \item the points of $L^\perp \setminus L = \bigcup_{i=1}^{q^e+1} \tilde{P}_i$,
 \item the points of $\bbP_d$ not in $L^\perp$. We denote this set by $Z$.
\end{enumerate}

Let $\sigma_i$ be a permutation (not necessarily collineation) of the hyperplanes of $\tilde{P}_i$ that 
preserves the parallel classes of $\tilde{P}_i$.
Let $z \in Z$. By Lemma \ref{lem:basic_perp_property} and $L \nsubseteq z^\perp$, $z^\perp \cap P_i$ is a $(d-1)$-subspace of $P_i$.
In particular, we can identify $z^\perp \cap P_i$ with a hyperplane of $\tilde{P}_i$ and let $\sigma_i$
act on it. Now we are ready to define our switched graph.

Let $\Gamma(L, \sigma_1, \ldots, \sigma_{q^e+1})$ be a graph with $X$ as vertices, where adjacency for two different vertices $x, y$ is defined as follows:
\begin{enumerate}[i.]
 \item If $x \in \tilde{P}_i$ and $y \in Z$, then $x$ and $y$ are adjacent if $x \in (y^\perp \cap \tilde{P}_i)^{\sigma_i}$.
 \item If $x \in Z$ and $y \in \tilde{P}_i$, then this follows by symmetry from the previous case.
 \item Otherwise, $x$ and $y$ are adjacent if $x \in y^\perp$.
\end{enumerate}
\end{construction}

\begin{rem}\label{rem:switching}
  Consider the case $q=2$. Here each parallel class contains only $2$ hyperplanes, so
  we can define $\sigma_i$ by $\sigma_i(H) = \tilde{P}_i \setminus H$.
  Then our construction corresponds to the construction described in \cite{Barwick2016}
  for $s=g-1$ (in the notation of \cite{Barwick2016}).
  So in this case our construction can be seen as a switching set in the usual sense.
\end{rem}

Throughout this section, we fix our choice for $L$ and all $\sigma_i$.
For the rest of this section we denote $\Gamma(L, \sigma_1, \ldots, \sigma_{q^e+1})$ by $\Gamma$.
Notice that $\Gamma(L, \id_1, \ldots, \id_{q^e+1})$ is the usual collinearity graph 
of the polar space, so $\Gamma_0 = \Gamma(L, \id_1, \ldots, \id_{q^e+1})$.
We will refer to adjacency in $\Gamma_0$ by \textit{$\Gamma_0$-adjacency} and adjacency in $\Gamma$ by \textit{$\Gamma$-adjacency}.

Our main result is the following.

\begin{theorem}\label{thm:construction}
  The graph $\Gamma(L, \sigma_1, \ldots, \sigma_{q^e+1})$ is strongly regular with the parameters $(v_0, k_0, \lambda_0, \nu_0)$.
\end{theorem}

We shall proof this theorem in several short lemmas.

\begin{lemma}\label{lem:cliques}
  A point $x \in \tilde{P}_i$ is not $\Gamma$-adjacent to any point of $\tilde{P}_j$ for $i \neq j$.
\end{lemma}
\begin{proof}
  Adjacencies between $\tilde{P}_i$ and $\tilde{P}_j$ do not change between $\Gamma$ and $\Gamma_0$.
  Suppose that there is a point $x \in \tilde{P}_i$ and a point $y \in \tilde{P}_j$ such that $\langle x, y \rangle$
  is a line of $\bbP_d$. Then $\langle L, x, y \rangle/L$ is a line of $L^\perp/L$.
  By Lemma \ref{lem:basic_quotient_geom}, $L^\perp/L \cong \bbP_{1}$, so $L^\perp/L$ contains no lines.
  This is a contradiction.
\end{proof}

\begin{lemma}\label{lem:one_point_per_line_case1}
  Let $\ell$ be a line of $\bbP_d$, which meets $\tilde{P}_j$ in a point, and let $x \in \tilde{P}_i$ with $i \neq j$. 
  Then $x$ is $\Gamma$-adjacent to exactly one point of $\ell$.
\end{lemma}
\begin{proof}
  Let $y$ be the intersection of $\ell$ with $\tilde{P}_j$.
  Then $L \subseteq y^\perp$, so $L' := \ell^\perp \cap L$ is a $(d-2)$-space or a $(d-1)$-space.
  If $L'$ is a $(d-1)$-space, then $\langle L, \ell \rangle$ is an isotropic $(d+1)$-space.
  This contradicts Lemma \ref{lem:gen_max_dimension}.
  Hence, $L'$ has dimension $d-2$ and corresponds to the parallel class $\scrC$ of hyperplanes of $\tilde{P}_i$
  which contain $L'$ at infinity.
  So for $z \in \ell$ with $z \neq y$ we have $z^\perp \cap \tilde{P}_i \in \scrC$
  and, more generally by Lemma \ref{lem:Hxyz_pw_dis}, $\scrC = \{ z^\perp \cap \tilde{P}_i: z \in \ell, z \neq y \}$.
  As $\scrC$ partitions the points of $\tilde{P}_i$ and $\sigma_i$ preserves parallel classes, 
  $x$ is $\Gamma$-adjacent to exactly one $z \in \ell$ with $z \neq y$.
\end{proof}

\begin{lemma}\label{lem:one_point_per_line_case2}
  Let $\ell$ be a line of $\bbP_d$, which meets $P_i$ in a point, and let $x \in \tilde{P}_i$. 
  If $x$ is $\Gamma$-adjacent to one point of $\ell \setminus P_i$, then $x$ is $\Gamma$-adjacent to all
	points of $\ell$.
\end{lemma}
\begin{proof}
  We have that $H := \ell^\perp \cap P_i$ has dimension $d-1$ or $d$.
  If $H$ is a $d$-space, then $\langle \ell, P_i \rangle$ is an isotropic $(d+1)$-space.
  This contradicts Lemma \ref{lem:gen_max_dimension}. Hence, $H$ has dimension $d-1$.
  If $H = L$, then the assertion is trivial.
  If $H \neq L$, then $H$ and $H^{\sigma_i}$ correspond to a hyperplanes of $\tilde{P}_i$.
  By Lemma \ref{lem:basic_perp_property}, if $x \in H^{\sigma_i}$, then $x$ is $\Gamma$-adjacent to all points of $\ell$.
  If $x \notin H$, then, by Lemma \ref{lem:basic_perp_property}, $x$ is not adjacent to any points of $\ell$ except $\ell \cap P_i$.
  This shows the assertion.
\end{proof}

\begin{lemma}\label{lem:two_hyps_in_Pi}
  Let $H$ and $H'$ be two different $(d-1)$-dimensional subspaces of a generator $P$.
  Then $H^\perp \cap H'^\perp = P$.
\end{lemma}
\begin{proof}
  Suppose to the contrary that there is a point $p \in H^\perp \cap H'^\perp$ which is not in $P$.
  Then $\langle p, H, H' \rangle$ is an isotropic subspace of dimension $d+1$.
  This contradicts Lemma \ref{lem:gen_max_dimension}.
\end{proof}

First we establish that the graph is regular. If a point $x$ lies in $L$, then none of the 
adjacencies of $x$ change in our construction. If $x$ lies in $Z$, then $x$ is 
$\Gamma_0$-adjacent to the vertices in $x^\perp \cap \tilde{P}_i$. Hence, $x$ is
$\Gamma$-adjacent to the vertices in $(x^\perp \cap \tilde{P}_i)^{\sigma_i}$ which
by definition of $\sigma_i$ contains the same number of vertices as $x^\perp \cap \tilde{P}_i$.
The only non-trivial case is handled in the following lemma.

\begin{lemma}\label{lem:k_eq_k0}
  Let $x \in \tilde{P}_i$. Then $x$ is $\Gamma$-adjacent to $k_0$ vertices.
\end{lemma}
\begin{proof}
  Let $\scrH$ denote the set of hyperplanes of $P_i$ with $x \in H^{\sigma_i}$
  for all $H \in \scrH$.
  Notice that $|\scrH|$ is the number of hyperplanes on $x$, so $|\scrH| = \gaussm{d-1}$.
  The vertex $x$ is $\Gamma$-adjacent to the vertices of $Z$ which are in 
  $\bigcup_{H \in \scrH} (H^\perp \setminus P_i)$. By Lemma \ref{lem:gens_on_comax}, 
  we have $q^e \cdot q^{d-1}$ points in $H^\perp \setminus P_i$. By Lemma \ref{lem:two_hyps_in_Pi},
  $x$ is $\Gamma$-adjacent to
  \begin{align*}
    |\scrH| q^{d-1+e} = \gaussm{d-1} q^{d-1+e} = k_0 - q\gaussm{d-1}.
  \end{align*}
  vertices of $Z$. In $L^\perp$ the vertex $x$ is $\Gamma$-adjacent to the $\gaussm{d}-1 = q\gaussm{d-1}$ vertices of $P_i$ except $x$ itself.
  The assertion follows.
\end{proof}

The following result implies constant parameters $\lambda_0$ and $\mu_0$.

\begin{lemma}\label{lem:lambda_mu_easy_cases}
  For $x, y \in L \cup Z \cup \tilde{P}_j$ the number of vertices in $\tilde{P}_i$, $i \neq j$
  that are $\Gamma$-adjacent to $x$ and $y$ equals the number of
  vertices in $\tilde{P}_i$ that are $\Gamma_0$-adjacent to $x$ and $y$.
\end{lemma}
\begin{proof}
  Vertices of $L$ are $\Gamma$- and $\Gamma_0$-adjacent to all vertices of $\tilde{P}_i$ by definition,
  vertices of $\tilde{P}_j$ are $\Gamma$- and $\Gamma_0$-adjacent to no vertices of $\tilde{P}_i$ by Lemma \ref{lem:cliques},
  and vertices of $Z$ are $\Gamma$- and $\Gamma_0$-adjacent to all vertices of an (affine) hyperplane of $\tilde{P}_i$.
  As ${\sigma_i}$ preserves parallel classes of hyperplanes, the assertion follows.
\end{proof}

\begin{lemma}\label{lem:lambda_L_vs_Pi}
  Let $x \in L$ and $y \in \tilde{P}_i$. Then $x$ and $y$ have $\lambda_0$ common neighbors in $\Gamma$.
\end{lemma}
\begin{proof}
  Let $\scrH$ be the set of hyperplanes $H$ of $P_i$ such that $x, y \in H^{\sigma_i}$.
  The number of $(d-2)$-dimensional subspaces through $x$ in $L$ is $\gaussm{d-2}$, so $x$ lies in $\gaussm{d-2}$ 
  parallel classes of hyperplanes of $\tilde{P}_i$. As $\sigma_i$ preserves parallel classes, for each parallel class
  there is exactly one $(d-1)$-dimensional subspace $H$ of $\tilde{P}_i$ such that $y \in H$.
  Hence, $|\scrH| = \gaussm{d-2}$.
  
  Exactly the vertices of $Z$ which are in 
  $\bigcup_{H \in \scrH} (H^\perp \setminus P_i)$ are $\Gamma$-adjacent to $x$ and $y$.
  By Lemma \ref{lem:gens_on_comax}, 
  we have $q^e \cdot q^{d-1}$ points in $H^\perp \setminus P_i$.
  By Lemma \ref{lem:two_hyps_in_Pi},
  \begin{align*}
    |\scrH| q^{d-1+e} = \gaussm{d-2} q^{d-1+e} = \lambda_0 - \gaussm{d} + 2
  \end{align*}
  vertices of $Z$ are $\Gamma$-adjacent to $x$ and $y$.
  As there are $\gaussm{d} - 2$ vertices in $P_i$ without $x$ and $y$, the assertion follows.
\end{proof}

Notice that most of our counting arguments go as the proofs of Lemma \ref{lem:k_eq_k0} and Lemma \ref{lem:lambda_L_vs_Pi}.

\begin{lemma}\label{lem:lambda_proof_Z_vs_P}
  Let $x \in Z$ and $y \in \tilde{P}_i$ such that $x$ and $y$ are $\Gamma$-adjacent.
  Then $x$ and $y$ have $\lambda_0$ common neighbors in $\Gamma$.
\end{lemma}
\begin{proof}
  Let $\scrH$ be the set of hyperplanes $H$ of $\tilde{P}$
  such that $y \in H^{\sigma_i}$.
  As $x$ and $y$ are $\Gamma$-adjacent, there exists one $H_0 \in \scrH$
  with $x^\perp \cap P_i = H_0$.
  By the definition of $\Gamma$-adjacency, $x$ and $y$ have $\gaussm{d}-2$
  common neighbors in $H_0^\perp$. 
  
  By Lemma \ref{lem:two_hyps_in_Pi}, 
  the points in $H^\perp \setminus P_i$ are non-$\Gamma$-adjacent to the points 
  of $H'^\perp \setminus P_i$ for all different $H, H' \in \scrH$.
  For $H \in \scrH$ with $H \neq H_0$ all points in $H^\perp$ are $\Gamma$-adjacent to $y$ by
  our definition of $\scrH$, while $x^\perp$ intersects all $q^e+1$ generators through $H$ in
  a hyperplane. Hence, $(x^\perp \cap H^\perp) \setminus P_i$ contains
  \begin{align*}
    q^{d-2+e}
  \end{align*}
  points. As $|\scrH| = \gaussm{d-1}$ and $\gaussm{d-1}-1 = q\gaussm{d-2}$, the number of common neighbors of $x$ and $y$ is
  \begin{align*}
    \gaussm{d}-2 + q\gaussm{d-2} \cdot q^{d-2+e} &= q^2\gaussm{d-2}+q-1 + q^{d-1+e} \gaussm{d-2}\\
    &= q-1 + q^2(q^{d-3+e}+1) \gaussm{d-2} = \lambda_0.
  \end{align*}
\end{proof}

\begin{lemma}\label{lem:lambda_proof_P_vs_P}
  Let $x, y \in \tilde{P}_i$ be $\Gamma$-adjacent.
  Then $x$ and $y$ have $\lambda_0$ common neighbors in $\Gamma$.
\end{lemma}
\begin{proof}
  The vertices $\Gamma$-adjacent to $x$ and $y$ in $L^\perp$ are exactly the 
  $\gaussm{d}-2$ vertices in $P_i$ which are not $x$ or $y$. The vertices $\Gamma$-adjacent to $x$ and $y$ in $Z$
  are exactly the vertices perpendicular to the $\gaussm{d-2}$ hyperplanes $H$ with $ x, y \in H^{\sigma_i}$.
  
  By Lemma \ref{lem:two_hyps_in_Pi}, each point $\Gamma$-adjacent to $x$ and $y$ is perpendicular to exactly one $H \in \scrH$.
  We have $H^\perp \setminus P_i = q^{d-1+e}$.
  Hence, the number of vertices $\Gamma$-adjacent to $x$ and $y$ in $Z$ is
  \begin{align*}
    \gaussm{d-2} q^{d-1+e}.
  \end{align*}
  As we have
  \begin{align*}
    \gaussm{d-2} q^{d-1+e} + \gaussm{d} - 2 = \gaussm{d-2} q^{d-1+e} + q^2\gaussm{d-2} + q - 1 = \lambda_0,
  \end{align*}
  the assertion follows.
\end{proof}

\begin{lemma}\label{lem:mu_P_vs_P}
  Let $x \in \tilde{P}_i$ and $y \in \tilde{P}_j$ with $i \neq j$.
  Then $x$ and $y$ have $\mu_0$ common neighbors in $\Gamma$.
\end{lemma}
\begin{proof}
  The vertices $\Gamma$-adjacent to $x$ and $y$ in $L^\perp$ are exactly the points in $L$ by Lemma \ref{lem:cliques}, so
  we have $\gaussm{d-1}$ of them.
  Next we count the vertices $\Gamma$-adjacent to $x$ and $y$ in $Z$. 
  Let $\scrH$ be the set of hyperplanes $H$ of $\tilde{P}_i$ with $x \in H^{\sigma_i}$.
  Recall that $x$ lies in $\gaussm{d-1}$ parallel classes of hyperplanes, so $|\scrH| = \gaussm{d-1}$.
  As $\bigcup_{H \in \scrH} (H^\perp \setminus P_i)$ contains 
  \begin{align*}
    \gaussm{d-1} q^{d-1+e}
  \end{align*}
  points, by Lemma \ref{lem:one_point_per_line_case2}, $x$ lies on $\gaussm{d-1} q^{d-2+e}$ lines of $\bbP_d$
  which meet $\tilde{P}_i$ in a point. By Lemma \ref{lem:one_point_per_line_case1},
  $y$ is adjacent to exactly one point on each of these lines. We obtain that
  the number of vertices that are $\Gamma$-adjacent to $x$ and $y$ is
  \begin{align*}
    &\gaussm{d-1} q^{d-2+e} + \gaussm{d-1} = \mu_0.
  \end{align*}
\end{proof}

\begin{lemma}\label{lem:mu_Z_vs_P}
  Let $x \in Z$ and $y \in \tilde{P}_i$ such that $x$ and $y$ are not $\Gamma$-adjacent.
  Then $x$ and $y$ have $\mu_0$ common neighbors in $\Gamma$.
\end{lemma}
\begin{proof}
  Let $\scrH$ the hyperplanes $H$ of $P_i$ such that $y \in H^{\sigma_i}$.
  The vertex $y$ lies on $\gaussm{d-1}$ hyperplanes, so $|\scrH| = \gaussm{d-1}$. 
  The vertices adjacent $x$ in $Z$ are $\bigcup_{H \in \scrH} (H^\perp \setminus P_i)$ and there are
  \begin{align*}
    \gaussm{d-1} q^{d-1+e}
  \end{align*}
  of them. 
  As $x$ and $y$ are not $\Gamma$-adjacent,
  no element of $\scrH$ corresponds to $x^\perp \cap P_i$.
  As $x^\perp$ intersects each of the $q^e+1$ generators through $H \in \scrH$ in a $(d-1)$-dimensional subspace
  not equal to $H$,
  we obtain that $\gaussm{d-1} q^{d-2+e}$ points in $Z$ are $\Gamma$-adjacent to $x$ and $y$.
  As $x$ is $\Gamma$-adjacent to the $(d-1)$-space $(x^\perp \cap P_i)^{\sigma_i}$ of $P_i$, there
  are $\gaussm{d-1}$ vertices in $L^\perp$ which are $\Gamma$-adjacent to $x$ and $y$.
  We obtain that $\mu_0$ vertices are $\Gamma$-adjacent to $x$ and $y$.
\end{proof}

\begin{proof}[of Theorem \ref{thm:construction}]
  Lemma \ref{lem:k_eq_k0} shows that each vertex in a $\tilde{P}_i$ has $k_0$ neighbors in $\Gamma$.
  For all other vertices the definition of $\Gamma$ makes it obvious that the degree does not change 
  between $\Gamma$ and $\Gamma_0$.
  
  The following table lists which lemmas show that the codegree of two adjacent vertices is $\lambda_0$.
  Here $i \neq j$.
  \begin{center}
\begin{tabular}{l|llll}
$x \setminus y$ & $L$ & $\tilde{P}_i$ & $\tilde{P}_j$ & $Z$ \\ \hline
$L$ & Lemma \ref{lem:lambda_mu_easy_cases} & Lemma \ref{lem:lambda_L_vs_Pi} &  & Lemma \ref{lem:lambda_mu_easy_cases}\\
$\tilde{P}_i$ & Lemma \ref{lem:lambda_L_vs_Pi} & Lemma \ref{lem:lambda_proof_P_vs_P} & Lemma \ref{lem:cliques} & Lemma \ref{lem:lambda_proof_Z_vs_P}\\
$Z$ & Lemma \ref{lem:lambda_mu_easy_cases} & Lemma \ref{lem:lambda_proof_Z_vs_P} &  & Lemma \ref{lem:lambda_mu_easy_cases}
  \end{tabular}
  \end{center}
  
  The following table lists which lemmas show that the codegree of two non-adjacent vertices is $\mu_0$.
  Here $i \neq j$.
  \begin{center}
\begin{tabular}{l|llll}
$x \setminus y$ & $L$ & $\tilde{P}_i$ & $\tilde{P}_j$ & $Z$ \\ \hline
$L$ &  &  &  & Lemma \ref{lem:lambda_mu_easy_cases}\\
$\tilde{P}_i$ &  &  & Lemma \ref{lem:mu_P_vs_P} & Lemma \ref{lem:mu_Z_vs_P}\\
$Z$ & Lemma \ref{lem:lambda_mu_easy_cases} & Lemma \ref{lem:mu_Z_vs_P} &  & Lemma \ref{lem:lambda_mu_easy_cases}
  \end{tabular}
  \end{center}
Therefore, $\Gamma$ is a strongly regular graph with the same parameters as $\Gamma_0$.
\end{proof}

\section{Non-Isomorphy}\label{sec:noniso}

The construction of Theorem \ref{thm:construction} is very general.
In this section we show for a very special, minimalist case of this construction that the obtained graphs
are non-isomorphic to $\Gamma_0$. This supports the idea that there are plenty 
non-isomorphic strongly regular graphs with the parameters $(v_0, k_0, \lambda_0, \mu_0)$.

\begin{lemma}\label{lem:triples_Gamma0}
  Let $x, y, z$ be three pairwise $\Gamma_0$-adjacent points of $\bbP_d$, $d \geq 3$.
  The number of points $\Gamma_0$-adjacent
  to $x$, $y$ and $z$ is either $\lambda_0-1$ or
  \begin{align*}
    q^2+q-2 + q^3 (q^{d-4+e}+1) \gaussm{d-3}.
  \end{align*}
\end{lemma}
\begin{proof}
  The points $x$, $y$ and $z$ are either all collinear or they span a $3$-space of $\bbP_d$.
  
  If $x$, $y$ and $z$ are collinear, then $\langle x, y \rangle = \langle x, y, z \rangle$.
  Hence $\langle x, y \rangle^\perp = \langle x, y, z \rangle^\perp$.
  Hence $x$, $y$ and $z$ are $\Gamma_0$-adjacent to all points $\Gamma_0$-adjacent
  to $x$ and $y$ except $z$. These are $\lambda_0-1$ points.
  
  If $x$, $y$ and $z$ span a $3$-space $\pi$, then the quotient geometry $\pi^\perp/\pi$
  is isomorphic to $\bbP_{d-3}$ by Lemma \ref{lem:basic_quotient_geom}. As $\bbP_{d-3}$ has
  \begin{align*}
    (q^{d-4+e}+1) \gaussm{d-3}
  \end{align*}
  points, $\pi^\perp \setminus \pi$ contains $q^3 (q^{d-4+e}+1) \gaussm{d-3}$ points.
  As $\pi$ has $q^2+q-2$ points without $x$, $y$ and $z$, we obtain the second number.
\end{proof}

For the rest of the section we assume that $d$ is at least $3$.
Let $\pi$ be a $(d-2)$-dimensional subspace of $P_1$ with $\pi \nsubseteq L$.
Let $\ell$ be a line in $\pi^\perp$ that meets $P_1$ trivially (the existence of $\ell$ can
be easily seen in the quotient geometry $\pi^\perp/\pi$). Notice that $\ell$ lies completely in $Z$.
Let $x, y, z \in \ell$ pairwise different points of $\ell$. Define $H_x$ as $x^\perp \cap P_1$,
$H_y$ as $y^\perp \cap P_1$ and $H_z$ as $z^\perp \cap P_1$.
By Lemma \ref{lem:Hxyz_pw_dis}, the hyperplanes $H_x$, $H_y$ and $H_z$ are pairwise different and
  \begin{align*}
    H_x \cap H_y \cap H_z = \pi \nsubseteq L.
  \end{align*}

\begin{lemma}\label{lem:Hxyz_prime_int_empty}
  Let $H_z'$ be a hyperplane in $P_1$ that is parallel to $H_z$ in $\tilde{P}_1$.
  Then $H_x \cap H_y \cap H_z'$ is empty in $\tilde{P}_1$.
\end{lemma}
\begin{proof}
  By Lemma \ref{lem:Hxyz_pw_dis}, $H_x \cap H_y \subseteq H_z$.
  As $H_z'$ is parallel to $H_z$, it follows that $H_x \cap H_y \cap H_z'$ is empty in $\tilde{P}_1$.
\end{proof}

We define a graph $\Gamma_1$ according to Theorem \ref{thm:construction} as
$\Gamma(L, \sigma_1, \id_2, \ldots, \id_{q^e+1})$. Here $\sigma_1$ is the mapping
that swaps $H_z$ with $H_z'$ and fixes all other hyperplanes of $\tilde{P}_1$.

\begin{lemma}\label{lem:triples_Gamma}
  The number of vertices that are $\Gamma_1$-adjacent to $x$, $y$ and $z$ is
  \begin{align*}
    \lambda_0 - 1 - q^{d-3}.
  \end{align*}
\end{lemma}
\begin{proof}
  As in the proof of Lemma \ref{lem:triples_Gamma0}, exactly $\lambda_0 - 1$
  vertices are $\Gamma_0$-adjacent to $x$, $y$ and $z$. The only difference between $\Gamma_0$ and $\Gamma_1$ occurs for the adjacencies
  between $\{x, y, z\}$ and $\tilde{P}_1$. By Lemma \ref{lem:Hxyz_prime_int_empty}, $\dim(H_x \cap H_y \cap H_z) = d-2$ and
  $H_x^{\sigma_1} \cap H_y^{\sigma_1} \cap H_z^{\sigma_1} = H_x \cap H_y \cap H_z'$ is empty. Hence the number of vertices
  in $\tilde{P}_1$ that are $\Gamma_1$-adjacent to $\{ x, y, z\}$ is $q^{d-3}$ less than the number
  of vertices in $\tilde{P}_1$ that are $\Gamma_0$-adjacent to $\{ x, y, z\}$.
\end{proof}

\begin{cor}
  The graphs $\Gamma_0$ and $\Gamma_1$ are non-isomorphic for $d \geq 3$.
\end{cor}
\begin{proof}
  In $\Gamma_0$ three pairwise adjacent vertices have either $\lambda_0-1$ or $q^2+q-2 + q^3 (q^{d-4+e}+1) \gaussm{d-3}$
  common neighbors. In $\Gamma$ we have, by Lemma \ref{lem:triples_Gamma}, three pairwise adjacent
  vertices that have $\lambda_0 - 1 - q^{d-3}$ common neighbors.
  It is easy to verify that $\lambda_0 - 1 - q^{d-3} \neq q^2+q-2 + q^3 (q^{d-4+e}+1) \gaussm{d-3}$ for $d \geq 3$.
\end{proof}

\begin{rem}
 This shows that the graphs constructed here are non-isomprhpic to the graphs
 constructed by Kantor \cite{Kantor1982a} as in Kantor's construction three 
 vertices have either $\lambda_0-1$ or $q^2+q-2 + q^3 (q^{d-4+e}+1) \gaussm{d-3}$ 
 common neighbours.
\end{rem}

\section{Open Problems}\label{sec:openproblems}

Several questions arose during this work. Here are some of them.

Similar arguments as in Section \ref{sec:noniso} for pairwise collinear $k$-tuples instead of $3$-tuples
can show that we obtain at least $d-1$ pairwise non-isomorphic graphs.
We have $(q!)^{\gaussm{d-1}}$ choices for each $\sigma_i$ in $(\sigma_1, \ldots, \sigma_{q^e+1})$.
Clearly, these will not be all different graphs, but many will be, so in the author's opinion
any reasonable lower bound on the number of non-isomorphic graphs with the same parameters as $\Gamma_0$
is something like $q^{q^{d-2}}$.

\begin{question}
  Do we have at least $q^{q^{d-2}}$ pairwise non-isomorphic graphs with the same parameters 
  as $\bbP_d$? Can we do better?
\end{question}

A technique similar to \cite[Section 3]{Jungnickel2010} might be helpful for showing this.
Alternatively, one could try to count the common neighbors of triples in more detail.

There exists a strongly regular graph that is a collinearity graph of the building of type $E_6$ \cite[Section 10.8]{Brouwer1989}.
As our basic idea is to mess with a subgeometry of $\bbP_d$, it is tempting
to find a similar substructure in $E_6$.

\begin{question}
  Can one use a similar method to show that there are many strongly regular graphs
  with the same parameters as the collinearity graph of $E_6$?
\end{question}

There are other strongly regular graphs which come from similar geometric constructions. For these one can ask similar questions.
A natural generalization of strongly regular graphs are distance-regular graphs \cite{Dam2006}. 
Many of the known infinite families come from geometries such as the dual polar graph \cite[Section 9.4]{Brouwer1989}.
Van Dam and Koolen constructed new distance-regular graphs, the twisted Grassmann graphs, by applying a switching operation
to the Grassmann graphs \cite{Dam2005,Munemasa2015}.

\begin{question}
  Can one use similar techniques to show that there are many distance-regular graphs with
  the same parameters as (some) dual polar graphs?
\end{question}

The used technique generalizes Godsil-McKay switching for a very specific case.
This switching is a far more general method described in terms of adjacency matrices.

\begin{question}
  Can one formulate a generalization of the Godsil-McKay switching which covers the technique used here?
\end{question}

\paragraph*{Acknowledgment} I would like to thank Aida Abiad for getting me interested in the topic and answering many of my questions about it.
The idea for the described construction was triggered by the switching sets defined by Susan Barwick, 
Wen-Ai Jackson and Tim Penttila. I would like to thank them for this and for providing me with a preprint of their work.
Bill Kantor kindly pointed out that my construction is non-isomorphic to his (when it exists).

%

\end{document}